\documentclass{amsart}
\usepackage[T1]{fontenc} 
\usepackage{graphicx}

\usepackage{amsmath,amssymb}  
	\usepackage[all]{xy}
	\usepackage{float}
	
	\renewcommand{\t}{\tau}
	\newcommand{\Real}{\mathbb{R}}
	\newcommand{\Complex}{\mathbb{C}}
	\newcommand{\Integer}{\mathbb{Z}}
	\newcommand{\CJ}{\mathcal{CJ}}
	
	\newcommand{\dM}{\partial M}
	\newcommand{\DT}[1]{#1 \dots #1}
	\newcommand{\bydef}{\stackrel{\mbox{\tiny def}}{=}}
	\def \lmod#1\rmod {\left|\smash{#1}\right|\vphantom{#1}}
	\newcommand{\CSum}[2]{\operatorname{\mathcal C}\nolimits_{#1 \mid #2}}
	\newcommand{\pder}[2]{\frac{\partial #1}{\partial #2}}
	\newcommand{\pdertwo}[3]{\frac{\partial^2 #1}{\partial #2 \partial #3}}
	\newcommand{\name}[1]{\operatorname{\mathrm{#1}}}
	
	\def \lmod#1\rmod {\left|\smash{#1}\right|\vphantom{#1}}
	
	\def \Fock {\Lambda^{\infty/2}}
	\def \HInteger {\Integer + \frac12}
	
	\theoremstyle{plain}
	\newtheorem{theorem}{Theorem}
	\newtheorem{lemma}[theorem]{Lemma}
	\newtheorem{proposition}[theorem]{Proposition}
	
	\newtheorem{corollary}[theorem]{Corollary}
	\theoremstyle{definition}
	\def \definitionName {Definition}
	\newtheorem{definition}[theorem]{\definitionName}

	\theoremstyle{remark}
	\def \remarkName {Remark}
	\newtheorem{remark}[theorem]{\remarkName}

	\newtheorem{example}[theorem]{Example}

	\numberwithin{theorem}{section}
	\numberwithin{equation}{section}

	\author{Rapha\"el Fesler}
	\address{Higher School of Economics, Moscow, Russia}
	\email{raphael.fesler@gmail.com}
	
	\def \Norm {\name{Norm}}

	\title{Hurwitz numbers for reflection groups $B$ and $D$}

\begin{document}
		
		\begin{abstract}
			We are building a theory of simple Hurwitz numbers for the
			reflection groups B and D parallel to the classical theory for the
			symmetric group. We also study analogs of the cut-and-join
			operators.  An algebraic description of Hurwitz numbers and an
			explicit formula for them in terms of Schur polynomials are
			provided. We also relate Hurwitz numbers for B and D to ribbon
			decomposition of surfaces with boundary --- a similar result for the
			symmetric group was proved earlier by Yu.Burman and the author.
			Finally, the generating function of B-Hurwitz numbers is shown to
			give rise to two independent $\tau$-function of the KP hierarchy.
		\end{abstract}
		
		\maketitle
		
		\setcounter{tocdepth}{3}
		\tableofcontents
		
		\section*{Introduction}
		
		Hurwitz numbers are a classical topic in combinatorics and algebraic
		geometry; they date back to the work by Adolf Hurwtiz \cite{AdHurwitz}
		(1890). A classical definition of the Hurwitz number $h_{m,\lambda}$
		where $m$ is a non-negative integer and $\lambda=(\lambda_1 \DT\ge
		\lambda_s)$ is a partition, is: $h_{m,\lambda}$ is the number of
		sequences of $m$ transpositions $(\sigma_1 \DT, \sigma_m)$ such that
		their product $\sigma_1 \DT\circ \sigma_m$ belongs to a conjugacy
		class $C_\lambda \in S_n$ in the symmetric group $S_n$. By a classical
		result \cite{GouldJacks}, generating function of the Hurwitz numbers
		satisfies a second order parabolic PDE called the cut-and-join
		equation. Its right-hand side has Schur polynomials as eigenvectors,
		so the Hurwitz numbers can be expressed via them
		\cite{LandoKazarian}. The result by A.\,Okounkov about Toda lattice
		\cite{Okounkov} implies that the same generating function is a
		solution of the KP hierarchy (the proof uses tools coming from
		mathematical physics, namely the boson-fermion
		correspondence). Finally, in a recent work \cite{TwistedHurwitz} by
		Yu.Burman and the author Hurwitz numbers are shown to solve a problem
		in low-dimensional topology: they are the numbers of ways to glue $m$
		ribbons to a collection of $n$ disks so as to obtain a surface with a
		prescribed structure of the boundary.
		
		In this paper we show that there is a parallel theory of Hurwitz
		numbers for the reflection groups of series B and D. Section
		\ref{Sec:BD} is the starting point where we recall the classical
		embedding of $B_n$ and $D_n$ into the symmetric group $S_{2n}$; the
		reflections are mapped to transpositions and products of pairs of
		transpositions commuting with the involution $\t = (1,n+1) (2,n+2)
		\dots (n,2n)$ (for the group $D_n$, only pairs of transpositions are
		used). Also we recall a description of the conjugacy classes in these
		groups, taken from \cite{Carter} (see Proposition \ref{Pp:ConjClass};
		the classes are generally indexed by pairs of partitions $\lambda$ and
		$\mu$ such that $\lmod \lambda\rmod + \lmod \mu\rmod = n$).
		
		Section \ref{Sec:HurwNum} gives the actual definition of Hurwitz
		numbers for the groups $B_n$ and $D_n$: it is similar to the classical
		one with reflections instead of transpositions. For the B series, we
		count reflections of two classes (transpositions and pairs of
		transpositions) separately.
		
		In Sections \ref{Sec:Expl} and \ref{Sec:CJD} we study how the
		multiplication by a reflection affects conjugacy classes; this gives
		us an expression for the cut-and-join operator for the groups of
		series B and D.  After a suitable change of variables we see that the
		operator obtained for the B group is actually a tensor square of the
		classical cut-and-join. The cut-and-join for the D group is a direct
		sum of the cut-and-join for B and the classical cut-and-join
		rescaled. This leads to an expression of the Hurwitz numbers for B and
		D via Schur polynomials.
		
		In Section \ref{SSec:Ribbon} we provides a model for B-Hurwitz
		numbers involving ribbon decomposition of surfaces equipped with an
		involution. Finally, in Section \ref{SSec:Fermion} we apply the
		boson-fermion correspondence to the tensor square of the cut-and-join
		described above; it allows one to show that the generating function of
		the B-Hurwitz numbers is a $2$-parameter family of $\tau$-functions of
		the KP hierarchy, independently in $2$ variables, and to prove a
		similar result for D.
		
		\subsection*{Acknowledgements}
		
		The author wants to thank his advisor Yurii Burman for many 
		fruitful discussions and constant attention to this work.
		
		The research was partially funded by the HSE University Basic Research
		Program and by the International Laboratory of Cluster Geometry NRU
		HSE (RF Government grant, ag.~No.~075-15-2021-608 dated 08.06.2021)
		
		\section{Reflection groups $B_n$ and $D_n$}\label{Sec:BD}
		\subsection{Definition and embedding into $S_{2n}$}
		
		Let $B_n$ be a rank $n$ finite reflection group with the mirrors
		$x_i-x_j = 0$, $x_i+x_j = 0$ ($1 \le i < j \le n$) and $x_i = 0$ ($1
		\le i \le n$); the corresponding reflections are denoted $s_{ij}^+$,
		$s_{ij}^-$, and $\ell_i$, respectively (see \cite{Humphreys} for details
		and standard notation). The subgroup $D_n \subset B_n$ is generated by
		$s_{ij}^{\pm}$ only.
		
		Consider an involution $\t \in S_{2n}$ without fixed points. All such
		involutions are conjugate, and the results that follow do not depend
		on a particulate choice; still, fix $\t \bydef (1,n+1)(2,n+2) \dots
		(n,2n)$ (that is, $\t(i) = i+n \bmod 2n$ for all $i$).
		
		Groups $B_n$ and $D_n$ admit embedding into the permutation group
		$S_{2n}$. Namely,
		
		\begin{proposition}\label{Pp:Embed}
			There exists an embedding $\Phi: B_n \to S_{2n}$ such that
			$\Phi(s_{ij}^+) = (ij)(i+n,j+n)$, $\Phi(s_{ij}^-) = (i,j+n)(i+n,j)$
			and $\Phi(\ell_i) = (i,i+n)$. The image $\Phi(B_n)$ is the normalizer
			of $\tau$:
			\begin{equation*}
				\Phi(B_n) = \{\sigma \in S_{2n} \mid \sigma\tau = \tau\sigma\}
				\bydef \Norm(\tau),
			\end{equation*}
			and the image of the subgroup $D_n \subset B_n$ under $\Phi$ is the
			intersection of $\Norm(\tau)$ with group $S_{2n}^+  
			\subset S_{2n}$ of even permutations.  
		\end{proposition}
		
		Introduce a notation more convenient for our purposes: let $r_{ij}
		\bydef (i,j)(i+n,j+n)$ for all $1 \le i \ne j \le 2n$, $i \ne j\pm n$;
		addition is modulo $2n$. So if $1 \le i < j \le n$ then $r_{ij} =
		r_{i+n,j+n} = \Phi(s_{ij}^+)$ and $r_{i,j+n} = r_{i+n,j} =
		\Phi(s_{ij}^-)$. Also denote $l_i \bydef (i,i+n)$ for all $1 \le i \le
		2n$; so, $l_i = l_{i+n} = \Phi(\ell_i)$ if $i \le n$.
		
		A proof (rather elementary) of Proposition \ref{Pp:Embed} is preceded
		by an explicit description of $\Norm(\tau)$, which will be used
		extensively throughout this article:
		
		\begin{lemma}\label{Lm:AlphaBeta}
			If $x \in \Norm(\tau)$ then for every cycle $(\gamma_1 \dots \gamma_m)$
			of its cycle decomposition one of the following is true:
			
			\begin{enumerate}
				\item\label{It:Alpha} The cycle decomposition contains another
				cycle $(\gamma_1' \dots \gamma_m')$ of the same length such that
				$\gamma_i' = \tau(\gamma_i)$ for $i = 1 \DT, s$.
				\item\label{It:Beta} $s = 2t$ is even and $\gamma_{i+t} =
				\tau(\gamma_i)$ for all $i = 1 \DT, t$.
			\end{enumerate}
		\end{lemma}
		
		In the two cases we speak about {\em $\alpha$-pairs} of cycles and
		about {\em $\beta$-cycles}, respectively.
		
		\begin{proof}
			Let
			\begin{equation*}
				x = (\gamma_{11} \dots \gamma_{1m_1}) \dots (\gamma_{q1} \dots
				\gamma_{q m_q})
			\end{equation*}
			be the cycle decomposition. If $\tau x=x\tau$ then $\tau x \tau^{-1}
			= x$. So
			\begin{equation*}
				x = (\tau(\gamma_{11}) \dots \tau(\gamma_{1m_1})) \dots
				(\tau(\gamma_{q1}) \dots \tau(\gamma_{qm_q}))
			\end{equation*}
			The cycle decomposition is unique, so every cycle
			$(\tau(\gamma_{i1}) \dots \tau(\gamma_{im_i}))$ must be equal to some
			$(\gamma_{j_1} \dots \gamma_{j m_j})$. If these two cycles are
			different, then they form an $\alpha$-pair. If these
			two cycles are the same then there exists some $t$ such that
			$\tau(\gamma_k) = \gamma_{k+t\bmod m_i}$ for all $k = 1 \DT, m_i$. Since
			$\tau$ is an involution, one has $\gamma_k = \tau^2(\gamma_i)
			=\tau(\gamma_{k+t \bmod m_i}) = \gamma_{k+2t \bmod m_i}$ for all $k
			= 1 \DT, m_i$, which implies $m_i = 2t$, and this is a $\beta$-cycle.
		\end{proof}
		
		\begin{proof}[Proof of Proposition \ref{Pp:Embed}]
			Any element $x \in B_n$ is representable as $x = a_1\dots a_N$ where
			every $a_k$ is a reflection. Denote by $e_1 \DT, e_n$ the standard
			basis in $\Real^n$.
			
			\begin{lemma}\label{Lm:ActGen}
				For any $k = 1 \DT, n$ one has $x(e_k) = \pm e_\ell$ for some
				$\ell = 1 \DT, n$. If $x(e_k) = e_\ell$ then $(\tilde a_1 \DT
				\tilde a_N)(k) = \ell$, and if $x(e_k) = -e_\ell$ then
				$(\tilde a_1 \DT \tilde a_N)(k) = \t(\ell) (= \ell+n)$.
			\end{lemma}
			
			The lemma is proved by an immediate induction by $N$.
			
			Let us define now the map $\Phi: B_n \to S_{2n}$ as $\Phi(x) =
			\tilde a_1 \dots \tilde a_N$. If $x = b_1 \dots b_M$ is another
			representation of $x$ as a product of reflections, then Lemma
			\ref{Lm:ActGen} implies that $(\tilde b_1 \dots \tilde b_M)(k) =
			(\tilde a_1 \dots \tilde a_N)(k)$ for every $k$. Thus, the map
			$\Phi$ is well-defined and is a group homomorphism.
			
			Suppose $\Phi(x) = \name{id}$ (that is, $x$ belongs to the kernel of
			$\Phi$). Then $(\tilde a_1 \dots \tilde a_N)(k) = k$ for every $k =
			1 \DT, n$, and Lemma \ref{Lm:ActGen} asserts then that $x(e_k) =
			e_k$ for every $k$. Thus, $x = \name{id}$, so $\Phi$ is an
			embedding.
			
			Show now that $\Phi(B_n) = \Norm(\tau)$. Obviously, $\Phi(B_n) \subset
			\Norm(\tau)$ because $\Phi(s^\pm_{ij})$ and $\Phi(\ell_i)$ commute
			with $\tau$. Also, $\Phi(D_n) \subset S_{2n}^+$ because $r_{ij} \in
			S_{2n}^+$.
			
			Prove now that $\Norm(\tau) \subset \Phi(B_n)$. Let $x \in
			\Norm(\tau)$. By Lemma \ref{Lm:AlphaBeta}, $x$ is a product of
			$\alpha$-pairs and of $\beta$-cycles, so to prove that $x \in
			\Phi(B_n)$ it suffices to show that any $\alpha$-pair and any
			$\beta$-cycle are products of the elements $r_{ij}$ and
			$l_i$ ($1 \le i < j\le 2n$). This is the case:
			\begin{equation*}
				(i_1 \DT, i_m)(\tau(i_1) \DT, \tau(i_m)) = r_{i_1 i_2}
				\dots r_{i_{m-1} i_m}
			\end{equation*}
			and
			\begin{align*}
				(i_1 \DT, i_m \tau(i_1) \DT, \tau(i_m)) &= (i_1 \DT,
				i_m)(\tau(i_1) \DT, \tau(i_m)) (i_1\tau(i_1))\\
				&= r_{i_1 i_2} \dots r_{i_{m-1} i_m}
				l_{i_1}.
			\end{align*}
			To prove that $\Norm(\tau) \cap S_{2n}^+ \subset \Phi(D_n)$ observe
			that the element $x \in \Norm(\tau)$ is even if and only it contains
			an even number of $\beta$-cycles. (Indeed, an $\alpha$-pair is always
			an even permutation, while a $\beta$-cycle is odd because its length
			is even.) So it suffices to prove that the product of any two
			$\beta$-cycles is a product of $r_{ij}$ (for
			$\alpha$-cycles if was proved above). Indeed,
			\begin{align*}
				(i_1 \dots i_m \tau(i_1) &\dots \tau(i_m)) (j_1 \dots i_q
				\tau(j_1) \dots \tau(j_q))\\
				&= (i_1 \dots i_m)(\tau(i_1) \dots \tau(i_m)) (j_1 \dots i_q)
				(\tau(j_1) \dots \tau(j_q)) (i_1 \tau(i_1))(j_1 \tau(j_1)).
			\end{align*}
			The product of the first four cycles is already proved to be a
			product of $r_{ij}$, while for the two remaining cycles
			one has
			\begin{equation*}
				(i_1 \tau(i_1))(j_1 \tau(j_1)) = r_{ij} r_{i,\tau(j)}.
			\end{equation*}
		\end{proof}
		
		To save space, below we will not distinguish $B_n$ and $D_n$ from
		their images $\Phi(B_n), \Phi(D_n) \subset S_{2n}$.
		
		\subsection{Conjugacy classes}\label{SSec:ConjClass}
		
		Conjugacy classes in the permutation group $S_m$ are in one-to-one
		correspondence with partitions of $m$: two elements of $S_m$ are
		conjugate if and only if they have cycles of the same lengths
		(totalling $m$) in their cycle decomposition. A similar result for the
		reflection groups $B_n$ and $D_n$ can be found in \cite{Carter}. Quote
		it here for coompleteness; be warned that the original notation of
		\cite{Carter} differ from the one used here.
		
		Fix two partitions, $\lambda = (\lambda_1 \DT, \lambda_s)$ and
		$\mu = (\mu_1 \DT, \mu_t)$ such that $\lmod \lambda\rmod + \lmod
		\mu\rmod = \lambda_1 \DT+ \lambda_s + \mu_1 \DT+ \mu_t = n$, and
		consider a set $C_{\lambda\mid\mu}$ of elements $x \in B_n \subset
		S_{2n}$ such that their cycle decomposition contains
		\begin{enumerate}
			\item $\alpha$-pairs of lengths (each cycle) $\lambda_1, \lambda_2
			\DT, \lambda_s$;
			
			\item $\beta$-cycles of lengths $2\mu_1 \DT, 2\mu_t$ (recall that the
			length should be even).
		\end{enumerate}
		
		\begin{proposition}[\protect{\cite[Proposition 25]{Carter}}]\label{Pp:ConjClass}
			The set $C_{\lambda\mid\mu} \subset B_n$ is a conjugacy class. Every
			conjugacy class in $B_n$ is $C_{\lambda\mid\mu}$ for some $\lambda$ and
			$\mu$ such that $\lmod\lambda\rmod + \lmod\mu\rmod = n$.
		\end{proposition}
		
		For $D_n$ the answer is slightly more complicated. Take a partition
		$\lambda = (\lambda_1 \DT, \lambda_s)$ such that $\lmod\lambda\rmod =
		n$ and all $\lambda_i$ are even. For an element $\sigma \in
		C_{\lambda\mid\emptyset}$ write its cycle decomposition
		\begin{equation}\label{Eq:2LambdaEmpty}
			\sigma = (\gamma_{11} \dots  \gamma_{1\lambda_1})\dots (\gamma_{s1} \dots \gamma_{s\lambda_s})
			(\tau(\gamma_{11})
			\dots \tau(\gamma_{1\lambda_1})) \dots (\tau(\gamma_{s1})
			\dots \tau(\gamma_{s\lambda_s}))
		\end{equation}
		satisfying two conditions: cycles forming an $\alpha$-pair always
		stand in position $i$ and $i+s$ in the cycle decomposition,
		 and their matching elements $\gamma_{ij} $ and
		$\tau(\gamma_{ij} ) $ occupy the same positions in them. Now consider
		a permutation $\gamma = (\gamma_{11} \dots \gamma_{s\lambda_s} ) \in
		S_{2n}$ where the numbers $\gamma_{ij}$ are written exactly in the
		same order as in \eqref{Eq:2LambdaEmpty}. We write $\sigma \in
		C_{\lambda\mid\emptyset}^+$ if $\gamma$ is even, and $\sigma \in
		C_{\lambda\mid\emptyset}^-$ if $\gamma$ is odd. Representation of
		$\sigma$ in the form \eqref{Eq:2LambdaEmpty} is not unique but, as it
		is easy to see, the parity of $\gamma$ does not depend on a particular
		choice. (Recall, all the cycles in \eqref{Eq:2LambdaEmpty} have even
		length.)
		
		\begin{proposition}[\protect{\cite[Proposition 25]{Carter}}]\label{Pp:ConjD}\strut
			\begin{enumerate}
				\item If the partition $\mu$ contains an even number of parts then
				the conjugacy class $C_{\lambda\mid\mu} \subset B_n$ lies in $D_n$;
				if the number of parts is odd then $C_{\lambda\mid\mu}$ does not
				intersect $D_n$.
				
				\item If $\mu \ne \emptyset$ and the number of parts of $\mu$ is
				even then $C_{\lambda\mid\mu}$ is a conjugacy class in $D_n$.
				
				\item If $\lambda$ is a partition of $n$ containing at least one odd
				part $\lambda_i$ then $C_{\lambda\mid\emptyset}$ is a conjugacy class
				in $D_n$.
				
				\item If $\lambda$ is a partition of $n$ such that all its parts are
				even then $C_{\lambda \mid \emptyset}$ splits into two conjugacy
				classes in $D_n$, $C_{\lambda\mid\emptyset}^+$ and
				$C_{\lambda\mid\emptyset}^-$.    
			\end{enumerate}
			Any conjugacy class in $D_n$ is one of the classes listed above.
		\end{proposition}
		
		\begin{corollary}\label{Cr:ConjOdd}
			\begin{enumerate}
				\item Let $\lambda, \mu$ be partitions such that $\lmod\lambda\rmod
				+ \lmod\mu\rmod = n$, $\#\mu$ is even and either $\mu \ne
				\emptyset$ or at least one of the parts of $\lambda$ is odd. Then
				for any $x \in B_n$ and any $\sigma \in C_{\lambda\mid\mu} \subset
				D_n$ one has $x\sigma x^{-1} \in C_{\lambda\mid\mu}$. 
				\item Let $\lambda$ be a partition of $n$ such that all its parts
				are even, and let $x \in B_n$. Then for any $\sigma \in
				C_{\lambda\mid\emptyset}^+$ one has $x \sigma x^{-1} \in
				C_{\lambda\mid\emptyset}^+$ if $x \in D_n \subset B_n$ (that is,
				$x \in S_{2n}$ is an even permutation) and $x \sigma x^{-1} \in
				C_{\lambda\mid\emptyset}^-$ otherwise; the picture for $\sigma \in
				C_{\lambda\mid\emptyset}^-$ is symmetric.
			\end{enumerate}
		\end{corollary}
		
		In particular, $C_{{1^{n-2}2^1}\mid\varnothing} \subset D_n \subset
		B_n$ consists of all the reflections $r_{ij}$, and $C_{{1^{n-1}}\mid
			1}\subset B_n$, of all the reflections $l_i$.
		
		\section{Hurwitz numbers}\label{Sec:HurwNum}
		
		\subsection{Definitions}
		
		Fix a pair of partitions $\lambda, \mu$ with $\lmod\lambda\rmod +
		\lmod\mu\rmod = n$. Let $C_{\lambda\mid\mu} \subset B_n$ be the conjugacy
		class defined above.
		
		\begin{definition}\label{Df:HurwNumb}
			A sequence of reflections $(\sigma_1 \DT, \sigma_{m+\ell})$ of the
			group $B_n$ is said to have {\em profile} $(\lambda,\mu,m,\ell)$ if
			$\#\{p \mid \sigma_p = r_{ij}, 1 \le i < j \le 2n\} = m$,
			$\#\{p \mid \sigma_p = l_i, 1 \le i \le 2n\} = \ell$ and
			$\sigma_1 \dots \sigma_{m+\ell} \in C_{\lambda \mid \mu}$. The {\em
				Hurwitz numbers for the group $B_n$} are $h_{m,\ell,\lambda,\mu}^B
			= \frac{1}{n!} \#\{(\sigma_1 \DT, \sigma_{m+\ell}) \text{ is a
				sequence of profile } (\lambda,\mu,m,\ell)\}$.
		\end{definition}
		
		For the group $D_n$ we use the same numbers, in case they make sense:
		
		\begin{definition}\label{Df:D-Hurwitz}
			Let $m$ be a positive integer, and $\lambda$ and $\mu$, partitions
			where the number of parts $\#\mu$ is even. The Hurwitz number for
			the group $D_n$ is defined as $h_{m,\lambda,\mu}^D =
			h_{m,0,\lambda,\mu}^B$.
		\end{definition}
		
		In other words, $h_{m,\lambda,\mu}^D = \frac{1}{n!} \#\{(\sigma_1 \DT,
		\sigma_m) \text{ has profile } (\lambda,\mu,m)\}$ where $\sigma_p$ are
		reflections in the group $D_n$ (that is, $\sigma_p = r_{ij}$ for some
		$i$ and $j$), and the profile means that $\sigma_1 \dots \sigma_m \in
		C_{\lambda\mid\mu} \subset D_n$.
		
		\begin{remark}\label{Rm:DHurwRefine}
			Recall (Proposition \ref{Pp:ConjD}) that if $\mu = \emptyset$ and
			all the parts of the partition $\lambda$ are even then
			$C_{\lambda\mid\emptyset}$ splits into two conjugacy classes,
			$C_{\lambda\mid\emptyset}^+$ and $C_{\lambda\mid\emptyset}^-$.
			
			Denote by $\Sigma_{m,\lambda}^+$ and $\Sigma_{m,\lambda}^-$ the sets
			of $m$-tuples of reflections $\sigma_1 \DT, \sigma_m \in D_n$ such
			that $\sigma_1 \dots \sigma_m \in C_{\lambda\mid\emptyset}^+$
			(resp., $C_{\lambda\mid\emptyset}^-$). One can denote
			$h_{m,\lambda,\emptyset\pm}^D \bydef \frac{1}{n!}
			\#\Sigma_{m,\lambda}^\pm$, so that $h_{m,\lambda,\emptyset+}^D +
			h_{m,\lambda,\emptyset-}^D = h_{m,\lambda,\emptyset}^D$. It follows
			from Corollary \ref{Cr:ConjOdd}, though, that if $x \in B_n$ but $x
			\notin D_n$ then the map sending an $m$-tuple $(\sigma_1 \DT,
			\sigma_m)$ to $(x\sigma_1x^{-1} \DT, x\sigma_mx^{-1})$ is a
			bijection between $\Sigma_{m,\lambda}^+$ and
			$\Sigma_{m,\lambda}^-$. Thus, $h_{m,\lambda,\emptyset+}^D =
			h_{m,\lambda,\emptyset-}^D = \frac12 h_{m,\lambda,\emptyset}^D$, so
			considering $h_{m,\lambda,\emptyset\pm}^D$ makes little sense.
		\end{remark}
		
		Up to the end of this section we consider the Hurwitz numbers for the
		group $B_n$ only.
		
		Denote by $\CSum{\lambda}{\mu} \bydef \frac{1}{\#C_{\lambda\mid\mu}}
		\sum_{x \in C_{\lambda\mid\mu}} x \in \Complex[B_n]$ the normalized
		class sum. One has $\CSum{\lambda}{\mu} \in Z[B_n]$ (the
		center of the group algebra of $B_n$); by Proposition
		\ref{Pp:ConjClass}, $\CSum{\lambda}{\mu}$ form a basis in
		$Z[B_n]$. Consider now a ring of polynomials $\Complex[p,q]$
		where $p = (p_1, p_2, \dots)$ and $q = (q_1, q_2, \dots)$ are two
		infinite sets of variables. The ring is graded by the total degree
		where one assumes $\deg p_k=\deg q_k =k$ for all $k=1,2,\dots$. The
		map $\Psi$ defined by
		\begin{equation}\label{Eq:DefIsoB}
			\Psi(\CSum{\lambda}{\mu}) = p_\lambda q_\mu \bydef p_{\lambda_1}
			\dots p_{\lambda_s} q_{\mu_1} \dots q_{\mu_t}
		\end{equation}
		establishes an isomorphism between $Z[B_n]$ and the
		homogeneous component $\Complex[p,q]_n$ of total degree $n$.
		
		Now denote by
		\begin{equation}\label{Eq:DefT1}
			\mathcal T_1 := \frac12 \sum_{1 \le i < j \le 2n}
			r_{ij} = \#C_{2 1^{n-2}\mid\emptyset} \cdot \CSum{1^{n-2}2}{\emptyset}
		\end{equation}
		and
		\begin{equation*}
			\mathcal T_2 := \frac12 \sum_{1 \le i \le 2n} l_i = \#C_{1^{n-1}\mid1}
			\cdot \CSum{1^{n-1}}{1}
		\end{equation*}
		sums of all elements of the conjugacy classes containing reflections
		(recall that $r_{ij} = r_{i+n,j+n}$, and $l_i = l_{i+n}$, so every
		element is repeated twice in these sums; hence the factor
		$\frac12$). The elements $\mathcal T_1$ and $\mathcal T_2$ belong to
		$Z[B_n]$, so one can consider linear operators $T_1, T_2: Z[B_n] \to
		Z[B_n]$ of multiplication by $\mathcal T_1$ and $\mathcal T_2$,
		respectively. Obviously, $T_1$ and $T_2$ commute.
		
		Consider now linear operators $\CJ_1, \CJ_2: \Complex[p,q]_n \to
		\Complex[p,q]_n$ making the following diagrams commutative:
		
		\begin{equation}\label{Eq:CJViaGAlg}
			\xymatrix{
				Z[B_n] \ar[r]^{\times T_i} \ar[d]^\simeq &
				Z[B_n] \ar[d]^\simeq  \\
				\mathbb{C}[p,q]_n \ar[r]_{\CJ_i} &
				\mathbb{C}[p,q]_n   
			}, \qquad i = 1,2.
		\end{equation}
		
		Let now $\lambda, \lambda', \mu, \mu'$ be partitions such that
		$\lmod\lambda\rmod + \lmod\mu\rmod = \lmod\lambda'\rmod +
		\lmod\mu'\rmod = n$. Take an element $\sigma_* \in C_{\lambda\mid\mu}$
		and define the {\em multiplicity} $\langle \lambda, \mu \mid \lambda',
		\mu'\rangle_1$ as the number of reflections $u \in C_{2^11^{n-2} \mid
			\varnothing}$ (that is, $u = r_{ij}$) such that $u
		\sigma_* \in C_{\lambda'\mid\mu'}$; the multiplicity $\langle \lambda,
		\mu \mid \lambda', \mu'\rangle_2$ is defined in the same way with
		$u\in C_{1^{n-1}\mid1}$ (that is, $u = l_i$) instead.
		
		\begin{lemma}\label{Lm:Conj}
			Multiplicities do not depend on the choice of $\sigma_*$.
		\end{lemma}
		
		\begin{proof}
			Denote by $S_1(\sigma_*;\lambda',\mu') \subset
			C_{1^{n-2}2^1\mid\emptyset}$ the set of reflections $u$ such that
			$u\sigma_* \in C_{\lambda'\mid\mu'}$. Take $\sigma' = x\sigma_*
			x^{-1} \in C_{\lambda\mid\mu}$ where $x \in B_n$.  If $u \in
			S(\sigma';\lambda',\mu')$ then $u \sigma' = u x\sigma_* x^{-1} = x
			((x^{-1}ux) \sigma_*) x^{-1} \in C_{\lambda'\mid\mu'}$, which is
			equivalent to $(x^{-1}ux) \sigma_* \in C_{\lambda'\mid\mu'}$, that
			is, to $x^{-1}ux \in S(\sigma_*;\lambda',\mu')$. Thus, conjugation
			$u \mapsto x^{-1}ux$ is a one-to-one map sending
			$S_1(\sigma';\lambda',\mu')$ to $S_1(\sigma_*;\lambda',\mu')$, and
			therefore these two sets contain the same number of elements
			$\langle\lambda,\mu \mid \lambda',\mu'\rangle_1$. The reasoning for
			the multiplicity $\langle\lambda,\mu \mid \lambda',\mu'\rangle_2$ is
			the same.
		\end{proof}
		
		\begin{theorem}\label{Th:TViaMult}
			$T_i\CSum{\lambda}{\mu} = \sum_{\lambda',\mu'} \langle \lambda,\mu
			\mid \lambda',\mu'\rangle_i \CSum{\lambda'}{\mu'}$ for $i = 1,2$.
		\end{theorem}
		
		\begin{proof}
			\begin{equation}\label{Eq:ActionT}
				T_1 \CSum{\lambda}{\mu} = \frac{1}{\#C_{\lambda\mid\mu}}
				\sum_{\sigma \in C_{\lambda\mid\mu}} T_1\sigma =
				\frac{1}{\#C_{\lambda\mid\mu}} \sum_{\sigma \in
					C_{\lambda\mid\mu}}\sum_{u \in C_{1^{n-2}2^1\mid\varnothing}} u
				\sigma .
			\end{equation}
			It follows from Lemma \ref{Lm:Conj} that \eqref{Eq:ActionT} is an
			arithmetical mean of identical summands, so it is equal to each of
			them:
			\begin{equation*}
				T_1 \CSum{\lambda}{\mu} = \sum_{\lambda',\mu'} \sum_{\substack
					{u \in C_{1^{n-2}2^1 \mid \varnothing}\\
						u\sigma_* \in  C_{\lambda'\mid\mu'}}}
				u \sigma_* 
			\end{equation*}
			for any fixed $\sigma_* \in C_{\lambda\mid\mu}$. Using Lemma
			\ref{Lm:Conj} again, one obtains
			\begin{align*}
				T_1 \CSum{\lambda}{\mu} &= \sum_{\lambda',\mu'} \sum_{\substack{
						u \in C_{1^{n-2}2^1 \mid \varnothing}\\
						u\sigma_* \in C_{\lambda'\mid\mu'}}}
				\frac{1}{\#C_{\lambda'\mid\mu'}}
				\sum_{\tau \in C_{\lambda'\mid\mu'}} \tau\\
				&= \sum_{\lambda',\mu'} \#(\{u \in C_{1^{n-2}2^1 \mid\varnothing}:
				u\sigma_* \in C_{\lambda'\mid\mu'}\} \CSum{\lambda'}{\mu'}\\
				&= \sum_{\lambda',\mu'} \langle \lambda,\mu \mid
				\lambda',\mu'\rangle \CSum{\lambda'}{\mu'}.
			\end{align*}
			The proof for $T_2$ is the same with $C_{1^{n-1}\mid 1}$ instead of
			$C_{1^{n-2}2^1\mid \varnothing}$.
		\end{proof}
		
		\subsection{Generating function}
		
		Consider the following generating function for Hurwitz numbers of the
		group $B_n$:
		\begin{equation*}
			\mathcal{H}^B(\beta,\gamma,p,q) = \sum_{m,\ell}\sum_{\lambda,\mu}
			\frac{ h_{m,\ell,\lambda,\mu}^B}{m!\ell!}
			p_{\lambda}q_{\mu}\beta^m\gamma^\ell.
		\end{equation*}
		
		\begin{theorem}\label{Th:CJequ}
			$\mathcal{H}^B$ satisfies the {\em cut-and-join equations}
			\begin{equation}\label{Eq:CJ}
				\pder{\mathcal{H}^B}{\beta} = \CJ_1(\mathcal{H}^B) \quad \text{and}
				\quad \pder{\mathcal{H}^B}{\gamma} = \CJ_2(\mathcal{H}^B)
			\end{equation}
		\end{theorem}
		
		\begin{proof}
			Fix a positive integer $n$ and denote by $\mathcal{H}_n$ a degree
			$n$ homogeneous component of $\mathcal{H}^B$. The cut-and-join
			operators preserve the degree, so $\mathcal{H}^B$ satisfies the
			cut-and-join equations if and only if $\mathcal{H}_n$ does (for each
			$n$).
			
			Let $i$ be $1$ or $2$, and let
			\begin{equation*}
				\mathcal{G}_n \bydef \sum_{m,\ell \ge 0}
				\sum_{\lambda,\mu:\mid\lambda\mid +\mid\mu\mid = n}
				\frac{n!h_{m,\ell,\lambda,\mu}^{B}}{m!\ell!} \CSum{\lambda}{\mu}
				\beta^m \gamma^\ell \in\Complex[B_n]
			\end{equation*}
			An elementary combinatorial reasoning gives
			\begin{equation*}
				\mathcal{G}_n = \sum_{m\geq 0}
				\frac{\beta^m\gamma^\ell}{m!\ell!}T_1^m T_2^\ell(e_n)
			\end{equation*}
			where $e_n \in B_n$ is the unit element. Clearly
			\begin{equation*}
				T_1(\mathcal{G}_n) = \sum_{m\geq 0}
				\frac{\beta^m\gamma^\ell}{m!\ell!}(T_1)^{m+1}(e_{n}) = \sum_{m \ge
					1} \frac{\beta^{m-1}\gamma^\ell}{(m-1)!\ell!} (T_1)^{m}(e_{n}) =
				\pder{\mathcal{G}_n}{\beta}.
			\end{equation*}
			and, similarly, $T_2(\mathcal{G}_n) =
			\pder{\mathcal{G}_n}{\gamma}$. Applying the isomorphism $\Psi$ one
			obtains $\Psi T_1(\mathcal{G}_n) = \Psi(\pder{\mathcal{G}_n}{\beta})
			= \pder{}{\beta} \Psi(\mathcal{G}_n)$.  One has $\Psi(\mathcal{G}_n)
			= \mathcal{H}_n$, hence $\pder{}{\beta} \Psi(\mathcal{G}_n) =
			\pder{\mathcal{H}_n}{\beta}$; similarly, $\pder{}{\gamma}
			\Psi(\mathcal{G}_n) = \pder{\mathcal{H}_n}{\gamma}$. By the
			definition of the cut-and-join operators, $\Psi T_1(\mathcal{G}_n) =
			\CJ_1(\Psi(\mathcal{G}_n)) = \CJ_1({\mathcal{H}_n})$ and the same
			for $T_2$ and $\CJ_2$. Equalities \eqref{Eq:CJ} follow.
		\end{proof}
		
		\begin{corollary}\label{Cr:GenFun}
			\begin{equation}\label{Eq:GenFunc}
				\mathcal{H}^B(\beta,\gamma,p,q) =
				e^{\beta\CJ_1+\gamma\CJ_2}e^{p_1}
			\end{equation}
		\end{corollary}
		
		\begin{proof}
			It follows from Definition \ref{Df:HurwNumb} that
			$h^B_{0,0,\lambda,\mu} = \frac{1}{n!}$ if $\lambda = 1^n$ and
			$\mu=\emptyset$, and $h^B_{0,0,\lambda,\mu} = 0$ otherwise.  Thus
			$\mathcal{H}^B(0,0,p,q) = e^{p_1}$ and \eqref{Eq:GenFunc} follows
			from Theorem \ref{Th:CJequ}.
		\end{proof}
		
		\section{Explicit formulas}\label{Sec:Expl}
		
		\subsection{Multiplication of an element $x \in B_n \subset
			S_{2n}$ by a reflection}\label{SSec:CycleStruct}
		
		Let $\omega \in S_n$ and $1 \le a < b \le n$. The cycle structure of
		the product $\omega' = (ab)\omega$ depend on the cyclic structure of
		$\omega$ and positions of $a$ and $b$ as follows. If $a$ and $b$
		belong to the same cycle $(a \DT, u, b \DT, v)$ of $\omega$, then the
		cycle splits in $\omega'$ into two: $(a \DT, u)$ and $(b \DT, v)$ (``a
		cut''). If $a$ and $b$ belong to different cycles then the opposite
		thing happens: they glue together in $\omega'$ (``a join''). The
		cycles containing neither $a$ nor $b$ are the same for $\omega$ and
		$\omega'$.
		
		Let now $\omega \in B_n \subset S_{2n}$. The cyclic structure of $\omega'
		= \sigma\omega$ where $\sigma$ is a reflection (i.e.\ $\sigma =
		r_{ab} = (ab)(a+n,b+n)$ (addition modulo $n$) for $1 \le a
		< b \le 2n$, or $\sigma = l_a = (a,a+n)$ for $1\leq
		a\leq 2n$) depends on the cyclic structure of $\omega$ ($\alpha$-pairs
		and $\beta$-cycles) and the position of the points $a, b$ as shown in
		Table \ref{Fig:Mult} below.
		
		In this table, hexagons represent $\beta$-cycles and the pairs of
		triangles are $\alpha$-pairs; numbers inside denote the lengths of the
		cycles. The two-headed dashed arrows show the position of $a$ and $b$
		of the reflection $\sigma$. The solid arrows join cyclic structures of
		$\omega$ and $\sigma\omega$ for $\sigma = r_{ab}$, and the empty
		arrows, for $\sigma = l_a$ (recall that $\sigma$ is an involution,
		so the arrows are two-headed). For example, the first diagram shows
		the multiplication by $(a,b)(\t(a),\t(b))$ of a pair of $\beta$-cycles
		containing $a$ and $b$, respectively.
		
		The boxed number on the bottom left corner is the multiplicity
		$\langle\lambda,\mu \mid \lambda',\mu'\rangle$ between conjugacy
		classes containing $\omega$ and $\sigma\omega$.

		\begin{figure} 
			\includegraphics[scale=.6]{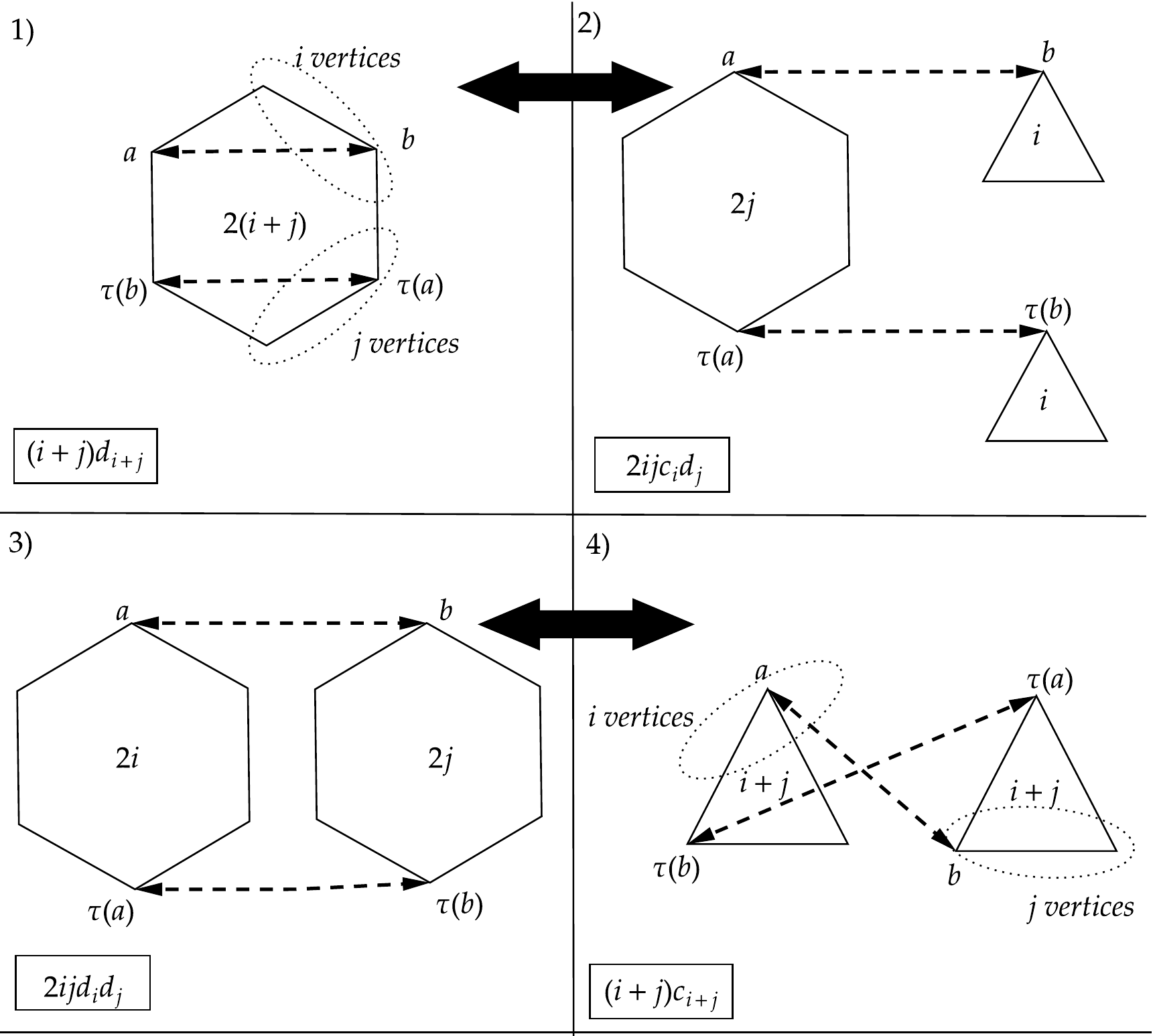}  
			\includegraphics[scale=.6]{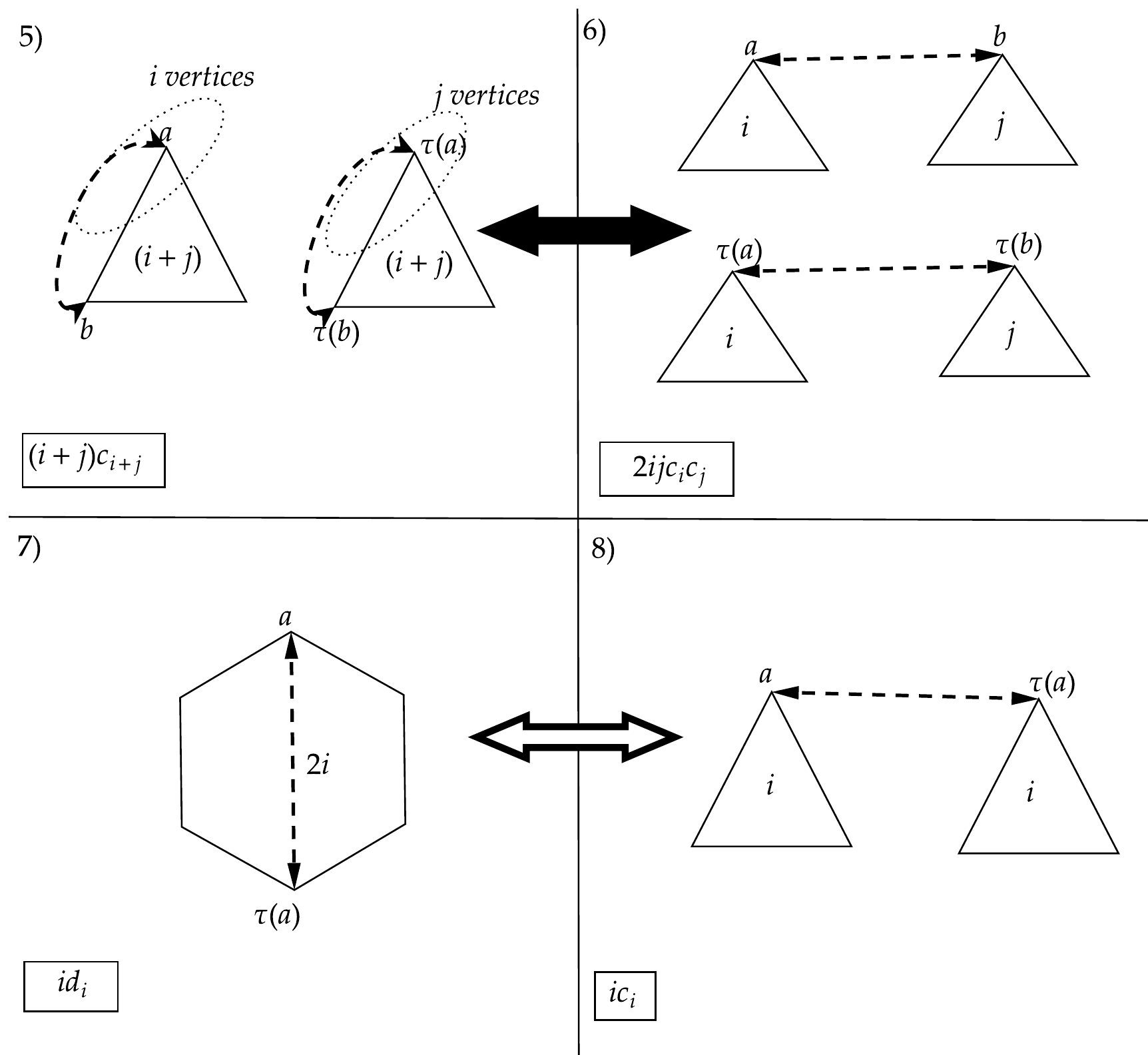}
			\caption{Multiplication of cycles by a reflection}\label{Fig:Mult}
		\end{figure}
		
		\subsection{Multiplicities}\label{SSec:Mult}
		
		Let $\omega \in C_{\lambda\mid\mu}$ where $\lambda = 1^{c_1} \dots n^{c_n}$
		and $\mu = 1^{d_1} \dots n^{d_n}$ --- that is, for every $k = 1 \DT,
		n$ the element $u$ contains $c_k$ $\alpha$-pairs of length $k$ and
		$d_k$ separate $\beta$-cycles of length $2k$. Calculate, for all
		$\lambda' = 1^{c_1'} \dots n^{c_n'}$ and $\mu = 1^{d_1'} \dots
		n^{d_n'}$, the multiplicities $\langle \lambda',\mu' \mid \lambda,
		\mu\rangle_1$ and $\langle \lambda',\mu' \mid \lambda,
		\mu\rangle_2$. In Section \ref{SSec:CycleStruct} we listed cases when
		this multiplicities may be nonzero: these are Cases
		\ref{It:OneBeta}--\ref{It:2join} below for $\langle\;\mid\;\rangle_1$ and
		Cases \ref{It:LCut} and \ref{It:LJoin} for $\langle\;\mid\;\rangle_2$.
		
		{\def \labelenumi {Case \theenumi.}
			
			\begin{enumerate}  
				
				\item\label{It:OneBeta} The number of possible positions for $a$ is
				the total number of elements in all the $\beta$-cycles of length
				$2(i+j)$, that is, $2(i+j) d_{i+j}$. Since one knows the
				lengths of $\alpha$-pairs in $\sigma \omega$, the position of $b$ is
				unique, once the position of $a$ is chosen; the same is for $\t(a)$
				and $\t(b)$. Doing like this, one counts every reflection $\sigma =
				r_{ab}$ twice, so that the multiplicity is $(i+j)d_{i+j}$.
				
				\item Again there are $2i c_i$ possible positions for $a$ and $2j d_j$
				possible positions for $b$; we are to divide by $2$ again by the
				same reason and the multiplicity is $2ijc_i d_j$.
				
				\item Like in Case \ref{It:OneBeta}, the multiplicity is $2ij d_i
				d_j$.  .
				
				\item\label{It:MatchAlphas} The number of possible positions for $a$
				is $2(i+j) c_{i+j}$; the position of $b$ is uniquely determined by
				the position of $a$ like in Case \ref{It:OneBeta}. Again, divide by
				$2$, to obtain the multiplicity $(i+j) c_{i+j}$.
				
				\item Same as case \ref{It:MatchAlphas}, the multiplicity is
				$(i+j) c_{i+j}$.
				
				\item\label{It:2join} The number of possible positions for $a$ is the
				total number of the elements of all the $\alpha$-pairs of length
				$i$, that is, $2i c_i$. Similarly, the number of possible positions
				for $b$ is $2j c_j$, so the multiplicity is $2ij c_i c_j$.
				
				\item\label{It:LCut} The number of possible positions for $a$ is
				$2i d_i$, for the same reason as above, we have to divide by $2$,
				so that the multiplicity is $i d_i$
				
				\item\label{It:LJoin} The number of possible positions for $a$ is
				$2i c_i$, as in the previous cases, we have to divide by $2$, so that
				the multiplicity is $i c_i$
			\end{enumerate}
		}
		
		Write now the terms of $\CJ_1$ and $\CJ_2$ explicitly. It follows from
		Theorem \ref{Th:TViaMult} that $\CJ_i p_\lambda q_\mu =
		\sum_{\lambda',\mu'} \langle\lambda,\mu\mid\lambda',\mu'\rangle_i
		p_{\lambda'} q_{\mu'}$ for $i\in \{1,2\}$.
		
		Let $\lambda',\mu'$ be as at diagram \ref{It:OneBeta} in Figure
		\ref{Fig:Mult}. The monomial $p_\lambda q_\mu$ contains
		$q_{i+j}^{d_{i+j}}$; The exponent at $q_{i+j}$ in the monomial
		$p_{\lambda'} q_{\mu'}$ is less by $1$, and the exponents at $p_i$ and
		$q_j$ are greater by $1$. We will have $q_{i+j} \pder{}{p_i}
		\pder{}{q_j}$ for this; to get a correct multiplicity, put the
		coefficient $(i+j)$ before.
		
		Similar reasoning for the remaining cases gives:
		
		\begin{multline}\label{Eq:CJ_1}
			\CJ_1 = \sum_{i,j=1}^\infty \biggl( (i+j) p_iq_j \pder{}{q_{i+j}} +
			2ij q_{i+j} \pdertwo{}{p_i}{q_j} + ijp_{i+j} \pdertwo{}{q_i}{q_j}\\
			+ \frac{1}{2} (i+j) q_iq_j \pder{}{p_{i+j}} + \frac{1}{2} (i+j)
			p_ip_j \pder{}{p_{i+j}} + ij p_{i+j} \pdertwo{}{p_i}{p_j}\biggr)
		\end{multline}
		and
		\begin{equation}\label{Eq:CJ_2}
			\CJ_2 = \sum_{i=1}^\infty \left(i p_i\frac{\partial}{\partial
				q_i} + i q_i\frac{\partial}{\partial p_i}\right)
		\end{equation}
		
		\subsection{Change of variables}\label{SSec:Change}
		
		In this section we reduce, by a suitable change of variables, the
		operators $\CJ_1$ and $\CJ_2$ to classical cut-and-join operators
		\begin{equation}\label{Eq:CJA}
			\CJ \bydef\sum_{i,j=1}^\infty \left(ijp_{i+j} \pdertwo{}{p_i}{p_j}
			+(i+j) p_ip_j \pder{}{p_{i+j}}\right)  
		\end{equation}
		and Euler fields
		\begin{equation}\label{Eq:Euler}
			E\bydef\sum_{i=1}^\infty ip_i \pder{}{p_i},
		\end{equation}
		respectively.
		
		\begin{proposition}\label{Pp:ChVar}
			Let $u_\ell = \frac{p_\ell+q_\ell}{2}$ and $v_\ell =
			\frac{p_\ell-q_\ell}{2}$. Then
			
			\begin{align*}
				\CJ_1 &= \sum_{i,j=1}^\infty \biggl( iju_{i+j}\pdertwo{}{u_i}{u_j} +
				(i+j) u_iu_j \pder{}{u_{i+j}} + ij v_{i+j} \pdertwo{}{v_i}{v_j}\\
				&\hphantom{\sum_{i,j=1}^\infty \biggl(} + (i+j) v_iv_j \pder{}{v_{i+j}}\biggr) = \CJ_u + \CJ_v,\\
				\CJ_2 &= \sum_{\ell=1}^\infty \ell \left(u_\ell \pder{}{u_\ell} -
				v_\ell \pder{}{v_{\ell}}\right) = E_u - E_v.
			\end{align*}
		\end{proposition}
		\noindent where by $\CJ_u$ and $\CJ_v$ we denote the operator \eqref{Eq:CJA}
		with $u_i$ (resp. $v_i$) substituted for $p_i$, and similarly, $E_u$
		and $E_v$.
		
		The proof is a direct computation.
		
		\begin{theorem}\label{Th:B-Schur}
			For all partitions $\lambda$ and $\mu$ the polynomials
			\begin{equation*}
				s_{\lambda\mid\mu}(p,q) \bydef s_\lambda((p+q)/2) s_\mu((p-q)/2)
			\end{equation*}  
			where $s_\lambda$ and $s_\mu$ are Schur polynomials and $(p \pm
			q)/2$ means $((p_1 \pm q_1)/2, (p_2 \pm q_2)/2, \dots)$, are
			eigenvectors of both $\CJ_1$ and $\CJ_2$; the respective eigenvalues
			are $\sum\limits_{i=1}^\infty \left(\lambda_i(\lambda_i-2i+1) +
			\mu_i(\mu_i-2i+1)\right)$ for $\CJ_1$ and $\sum\limits_{i=1}^\infty
			(\lambda_i - \mu_i)$ for $\CJ_2$.
		\end{theorem}
		
		The theorem follows immediately from Proposition \ref{Pp:ChVar} and
		the classical fact that Schur polynomials are eigenvectors of the
		cut-and-join operator \eqref{Eq:CJA} (see \cite{LandoKazarian}) and of
		the Euler field \eqref{Eq:Euler} ($s_{\lambda\mid\mu}$ are weighted
		homogeneous).
		
		Theorem \ref{Th:B-Schur} and the Cauchy identity \cite{Macdonald}
		imply that
		\begin{equation*}
			e^{p_1} = e^{u_1} e^{v_1} = \sum_{\lambda,\mu}
			s_{\lambda\mid\mu}(p,q) s_{\lambda\mid\mu}(1, 0, \dots; 1, 0, \dots)
		\end{equation*}
		
		This allows us to prove the following
		
		\begin{corollary}[of Theorem \ref{Th:B-Schur}]\label{Cr:LKGenFunc}
			\begin{align*}
				\mathcal{H}^B(\beta,\gamma,p,q) &= \sum_{\lambda,\mu}
				\exp\bigl(\beta \sum_{i=1}^\infty (\lambda_i(\lambda_i-2i+1) +
				\mu_i(\mu_i-2i+1) + \gamma \sum_{i=1}^{\infty}
				(\lambda_i-\mu_i)\bigr)\\
				&\times s_{\lambda\mid\mu}(1,0,\dots; 1,0,\dots)
				s_{\lambda\mid\mu}(p,q)
			\end{align*}
		\end{corollary}
		This is the B-analog of the formula expressing the classical Hurwitz
		numbers via the Schur polynomials \cite{LandoKazarian}
		
		Theorem \ref{Th:B-Schur} allows also to establish a direct relation
		between Hurwitz numbers of the groups $B_n$
		($h_{m,\ell,\lambda,\mu}^B$, studied here) and $S_n$ (the classical
		ones, denoted $h_{m,\lambda}^A$ till the end of this section).
		
		Namely, it follows from equation \ref{Eq:CJ_2} that $\CJ_2(p_\lambda
		q_\mu)= (|\lambda|+|\mu|)p_\lambda q_\mu$, and therefore, by the
		cut-and-join equation,
		\begin{multline*}
			\sum_{m,\ell}\sum_{\lambda,\mu}
			\frac{h_{m,\ell,\lambda,\mu}^B}{m!(\ell-1)!}
			p_{\lambda}q_{\mu}\beta^m\gamma^{\ell-1} = \CJ_2(\mathcal{H}^B) =
			\pder{\mathcal{H}^B}{\gamma}\\
			= \sum_{m,\ell}\sum_{\lambda,\mu} (\lmod\lambda\rmod +
			\lmod\mu\rmod) \frac{h_{m,\ell,\lambda,\mu}^B}{m!\ell!}
			p_{\lambda}q_{\mu}\beta^m\gamma^\ell.
		\end{multline*}
		This implies the equality 
		\begin{equation*}
			h_{m,\ell,\lambda,\mu}^B = (\lmod\lambda\rmod + \lmod\mu\rmod)^\ell
			h_{m,0,\lambda,\mu}^B
		\end{equation*}
		It is thus enough to focus on $h_{m,0,\lambda,\mu}^B$.
		
		By Corollary \ref{Cr:LKGenFunc}, $h_{m,0,\lambda,\mu}^B$ is the
		coefficient at the monomial $p_\lambda q_\mu\beta^m$ in the expression
		\begin{equation*}
			\sum_{\lambda',\mu',m_1,m_2}
			h_{m_1,\lambda'}^A(p+q)_{\lambda'}  
			h_{m_1,\mu'}^A(p+q)_{\mu'}\beta^{m_1+m_2}
			\frac{m!(\lmod\lambda\rmod + \lmod\mu\rmod)!}{2^{\#\lambda
					+\#\mu}\lmod\lambda'\rmod! m_1! \lmod\mu'\rmod! m_2!}
		\end{equation*}
		where $(p\pm q)_\lambda$ is understood as $(p_{\lambda_1}\pm
		q_{\lambda_1}) \dots (p_{\lambda_s}\pm q_{\lambda_s})$. The
		coefficient may be nonzero only if two conditions are satisfied:
		$m_1+m_2=m$ and $\lambda'+\mu'=\lambda+\mu$. Thus the relation can be
		rewritten as
		\begin{equation*}
			h_{m,0,\lambda,\mu}^B =
			\sum_{\substack{\lambda'+\mu'=\lambda+\mu\\m_1+m_2=m}}
			\frac{h_{m_1,\lambda'}^Ah_{m_2,\mu'}^A}{2^{\#\lambda +\#\mu}}
			\binom{m}{m_1}\binom{|\lambda|+|\mu|}{|\lambda'|} [p_\lambda q_\mu
			:(p+q)_{\lambda'}(p-q)_{\mu'}]
		\end{equation*}
		
		Now let $\lambda'=(1^{\alpha_1},2^{\alpha_2},\dots)$, $\mu'=(1^{\beta_1},2^{\beta_2},\dots)$ and
		$\mu=(1^{\gamma_1},2^{\gamma_2},\dots)$; it follows from the equality
		$\lambda+\mu = \lambda' + \mu'$ that $\lambda =
		(1^{\alpha_1+\beta_1-\gamma_1} 2^{\alpha_2+\beta_2-\gamma_2} \dots)$
		is fixed. If $f^{\gamma}_{\alpha\beta}$ is a coefficient at
		$x^\gamma$ at the polynomial $(1+x)^\alpha(1-x)^\beta$ then one has 
		
		\begin{equation*}
			h_{m,0,\lambda,\mu}^B =
			\sum_{\substack{\lambda'+\mu'=\lambda+\mu\\m_1+m_2=m}}
			\frac{h_{m_1,\lambda'}^A h_{m_2,\mu'}^A}{2^{\#\lambda + \#\mu}}
			\binom{m}{m_1}\binom{\lmod\lambda\rmod +
				\lmod\mu\rmod}{\lmod\lambda'\rmod} f^{\gamma_1}_{\alpha_1\beta_1}
			f^{\gamma_2}_{\alpha_2\beta_2}\dots
		\end{equation*}
		
		\section{Cut-and-join operator for the group $D_n$}\label{Sec:CJD}
		
		Let $Z[D_n]$ be a center of the group algebra of $D_n$; it is a vector
		space spanned by the class sums $\CSum{\lambda}{\emptyset}^+$ and
		$\CSum{\lambda}{\emptyset}^-$ where $\lmod\lambda\rmod = n$ and all
		the parts of $\lambda$ are even, and the class sums
		$\CSum{\lambda}{\mu}$ for all other $\lambda$ and $\mu$ such that 
		$\lmod\lambda\rmod + \lmod\mu\rmod = n$ and $\#\mu$ is
		even. The operator of multiplication by the element $\mathcal T_1 \in
		Z[D_1]$ defined by \eqref{Eq:DefT1} (a sum of all the reflections in
		$D_n$) acts on $Z[D_n]$; denote it $\Theta_D$ for brevity.
		
		Take an element $x \in B_n \setminus D_n$ (i.e.\ an odd permutation in
		$B_n$) and consider the map $\mathcal{I}:\Complex[D_n]\mapsto
		\Complex[D_n]$ defined by $\mathcal{I}(a) = xax^{-1}$. Corollary
		\ref{Cr:ConjOdd} implies that $\mathcal I$ maps the center $Z[D_n]$ to
		itself and its restriction to $Z[D_n]$ is an involution. This
		involution defines a splitting $Z[D_n] = V_n^+ \oplus V_n^-$ where
		$V_n^{\pm}$ are eigenspaces of $\mathcal I$ corresponding to
		eigenvalues $1$ and $-1$.
		
		Clearly, the vector space $V_n^+$ is spanned by the elements
		$\CSum{\lambda}{\mu}$ and $\CSum{\lambda}{\emptyset}^++
		\CSum{\lambda}{\emptyset}^-$, whereas the vector space $V_n^-$ is
		spanned by the the elements $\CSum{\lambda}{\emptyset}^+-
		\CSum{\lambda}{\emptyset}^-$. In particular, $V_n^-$ is nonempty if
		and only if $n = \lmod\lambda\rmod$ is even. Therefore, there exist a
		linear injection $\Phi_+: V_n^+ \to Z[B_n]$ and a linear
		isomorphism $\Phi_-: V_n^- \to Z[S_{n/2}]$ (for $n$ even) defined as
		\begin{align*}
			\Phi_+(\CSum{\lambda}{\mu}) &= \CSum{\lambda}{\mu}\\
			\Phi_+(\CSum{\lambda}{\emptyset}^+ + \CSum{\lambda}{\emptyset}^-) &= \CSum{\lambda}{\emptyset},
		\end{align*}
		and
		\begin{equation*}
			\Phi_-(\CSum{\lambda}{\emptyset}^+ - \CSum{\lambda}{\emptyset}^-) = C_{\lambda/2}
		\end{equation*}
		where $\lambda/2 = (\lambda_1/2 \DT, \lambda_s/2)$ (recall that
		$\CSum{\lambda}{\emptyset}^\pm$ are defined only if all the parts
		$\lambda_i$ of $\lambda$ are even). The image of $\Phi_+$ is a
		subspace $Z_e[B_n] \subset Z[B_n]$ spanned by all the class sums
		$\CSum{\lambda}{\mu}$ with $\#\mu$ even.
		
		\begin{theorem}\label{Th:CJD}
			\begin{enumerate}
				\item\label{It:Commutes} The operator $\Theta_D$ commutes with the
				involution $\mathcal I$, and therefore, $V_n^+$ and $V_n^-$ are
				$\Theta_D$-invariant.
				
				\item\label{It:CJDonVPlus} For the restriction of $\Theta_D$ on
				$V_n^+$ the following diagram is commutative:
				\begin{equation}\label{Eq:CJDViaGAlg}
					\xymatrix{
						V_n^{+} \ar[r]^{\Theta_D} \ar[d]^{\Phi_+} & V_n^{+} \ar[d]^{\Phi_+} \\
						Z[B_n] \ar[r]_{T_1} & Z[B_n]
					}
				\end{equation}
				
				\item\label{It:CHDonVMinus} If $n$ is even then for the restriction
				of $\Theta_D$ on $V_n^-$ the following diagram is commutative:
				\begin{equation}\label{Eq:CJDViaGAlgMinus}
					\xymatrix{
						V_n^- \ar[r]^{\Theta_D} \ar[d]^{\Phi_-} & V_n^- \ar[d]^{\Phi_-} \\
						Z[S_{n/2}] \ar[r]_{4T} & Z[S_{n/2}]
					}
				\end{equation}
			\end{enumerate}
		\end{theorem}
		
		\begin{corollary}
			$V_n^+$ is isomorphic to the space of homogeneous polynomials
			$\Complex[p,q]_n^e$ of total degree $n$ and even degree in
			$q$; the isomorphism is given by \eqref{Eq:DefIsoB}. $V_n^-$ is isomorphic to the space $\Complex[p]_{n/2}$ of
			homogeneous polynomials of total degree $n/2$, the isomorphism is
			given by
			\begin{equation*}
				\Phi_A(\CSum{\lambda}{\emptyset}^+-\CSum{\lambda}{\emptyset}^-) =
				p_{\lambda/2} \bydef p_{\lambda_1/2} \dots p_{\lambda_s/2}.
			\end{equation*}
			The following diagrams are commutative
			\begin{equation*}
				\xymatrix{
					V_n^+ \ar[r]^{\Theta_D} \ar[d]^{\Psi \circ \Phi_+} & V_n^+
					\ar[d]^{\Psi \circ \Phi_+} \\      
					\Complex[p,q]_n^e \ar[r]_{\CJ_1} & \Complex[p,q]_n^e
				}, \qquad
				\xymatrix{
					V_n^- \ar[r]^{\Theta_D} \ar[d]^{\Psi_A \circ \Phi_-} & V_n^-
					\ar[d]^{\Psi_A \circ \Phi_-} \\    
					\Complex[p]_{n/2} \ar[r]_{\CJ_A} & \Complex[p]_{n/2}
				}
			\end{equation*}
			where $\CJ_A$ is the classical cut-and-join operator expressed in
			rescaled variables $r_i = p_i/2$ and multiplied by $4$. 
		\end{corollary}
		
		In other words, the cut-and-join operator for the group $D_n$ is a
		direct sum of the cut-and-join operator for the group $B_n$,
		restricted to the space of polynomials of even degree in $q$, and, for
		$n$ even, the cut-and-join operator for the group $S_{n/2}$, expressed
		via rescaled variables $r_i = p_i/2$ and multiplied by $4$.
		
		\begin{proof}[Proof of Theorem \ref{Th:CJD} and the corollary]
			The image of $V_n^+ \subset Z[D_n] \subset \Complex[D_n]$ under the
			natural embedding $\iota: \Complex[D_n] \to \Complex[B_n]$ is
			$Z[B_n]$, and the restriction of $\iota$ to $V_n^+$ coincides with
			$\Phi_+$. So, assertion \ref{It:Commutes} follows from the fact that
			$\Theta_D$ is an operator of multiplication by the element $T_1 \in
			Z[D_n]$ such that $\iota(T_1) \in Z[B_n]$. Assertion
			\ref{It:CJDonVPlus} is then evident.
			
			Instead of proving assertion \ref{It:CHDonVMinus} we prove the
			commutativity of the second diagram from the corollary; it is
			obviously an equivalent fact. The description above implies that
			$\CJ_A = U \CJ V$ where the operator $V: \Complex[r] \to
			\Complex[p,q]$ is given by $(Vf)(p,q) = f(p_2, p_4, \dots)$ and $U:
			\Complex[p,q] \to \Complex[p]$, by $(Uh)(r) =
			\left.h\right|_{\substack{p_1=p_3 \DT= 0\\ p_2=r_1, p_4=r_2, \dots
					\\ q_1=q_2 \DT= 0}}$. Now the required formula follows directly
			from \eqref{Eq:CJ_1} and \eqref{Eq:CJA}.
		\end{proof}
		
		Hurwitz numbers for the group $D_n$ coincide with (some of) the
		Hurwitz numbers for the group $B_n$ by Definition \ref{Df:D-Hurwitz};
		see also remark \ref{Rm:DHurwRefine}.
		
		\section{Other properties of B-Hurwitz numbers and B-cut-and-join operator}
		
		Classical Hurwitz numbers $h_{m,\lambda}$ exhibit many interesting
		properties. They play an active part in the algebro-geometric theory
		of moduli spaces of complex curves and holomorphic functions on them
		(see \cite{GouldJacks,ELSV}), appear in some problems of
		low-dimensional topology \cite{BurmanZvonkine,TwistedHurwitz} and the
		theory of integrable systems \cite{Kramer,KazarianHodge,FSpace}. The
		classical cut-and-join operator also has various interpretations
		besides the one described here, including the Fock space
		representation (see \cite{Johnson} and the references therein).
		
		In this section we describe analogs of some of these results for the
		B-Hurwitz numbers and B-version of the cut-and-join operator. Proper
		references and short explanation of the classical picture is given in
		the beginning of each subsection.
		
		\subsection{Ribbon decomposition}\label{SSec:Ribbon}
		
		In this section we will be using notations and definitions from
		\cite{TwistedHurwitz}; to make the article self-contained, we quote
		here the most important definitions from there. A {\em
			decorated-boundary surface} (DBS) with $n$ marked points is a
		surface with the boundary $\dM$ and the marked points $a_1 \DT, a_n
		\in \dM$ satisfying some nontriviality conditions (see
		\cite{TwistedHurwitz} for details). All DBS in this section are
		assumed oriented.
		
		If $M$ is an oriented DBS with $n$ marked points then one defines its
		{\em boundary permutation} $\Sigma(M) \in S_n$ as follows:
		$\Sigma(M)(i) = j$ if the marked points $a_i$ and $a_j$ belong to the
		same connected component of the boundary $\dM$ and $a_j$ immediately
		follows $a_i$ as one moves around the boundary in the positive
		direction (according to the orientation).
		
		A {\em ribbon} is a long narrow rectangle $\rho$ with alternating
		black and white vertices; black vertices are joined by a diagonal. For
		a DBS $M$ and $1 \le i < j \le n$, denote by $G[i,j] M$ the result of
		gluing of a ribbon to $M$ such that the black vertices are identified
		with the marked points $a_i$ and $a_j$, and the short sides are glued
		to short segments of $\dM$ containing $a_i$ and $a_j$ and directed
		according to the orientation of $\dM$. Apparently, $G[i,j] M$ is an
		oriented DBS with the same marked points; we call $G[i,j]$ an
		operation of {\em ribbon gluing}.
		
		A {\em ribbon decomposition} of the DBS $M$ is an
		orientation-preserving diffeomorphism of $M$ with the DBS $G[i_m,j_m]
		\dots G[i_1,j_1] E_n$ obtained by consecutive ribbon gluing to a union
		$E_n$ of $n$ disks each one having one marked point $a_i$ on its
		boundary. The image in $M$ of the $k$-th ribbon glued is also called a
		ribbon (number $k$) and is denoted $\rho_k \subset M$.
		
		Now fix positive integers $m$ and $\ell$. Let $M$ be an oriented
		decorated-boundary surface with $2n$ marked points $a_1 \DT, a_{2n}$,
		and let $T: M \to M$ be an orientation-preserving smooth involution
		mapping marked points to marked points and having no fixed points on
		the boundary $\dM$. If $T(D) = D$ where $D \subset \dM$ is a connected
		component then $D$ contains an even number $2\mu_i$ of marked
		points. If $D$ is not $T$-invariant then $T$ exchanges it with another
		component $D'$; both contain the same number $\lambda_i$ of marked
		points. Numbers $\lambda_1 \DT, \lambda_s$ and $\mu_1 \DT, \mu_t$ form
		two partitions denoted $\Lambda_+(M,T)$ and $\Lambda_-(M,T)$,
		respectively; one has $\lmod\Lambda_+(M,T)\rmod +
		\lmod\Lambda_-(M,T)\rmod = n$.
		
		Let now $R$ a ribbon decomposition of $M$ and suppose that it is
		$T$-invariant: for any ribbon $\rho$ its image $T(\rho)$ is also a
		ribbon. Some ribbons are $T$-invariant: $T(\rho) = \rho$, while others
		come in pairs $\rho_1 = T(\rho_2), \rho_2 = T(\rho_1)$ (for brevity,
		call these $T$-pairs). If the diagonal of a ribbon $\rho$ joins marked
		points $a_i$ and $a_j$ then the diagonal to the ribbon $T(\rho)$ joins
		$T(a_i)$ and $T(a_j)$. In particular, diagonals of $T$-invariant
		ribbons join marked points $a_i$ and $T(a_i)$, for some $i$.
		
		A ribbon is homeomorphic to a disk, so, by the Brouwer theorem, $T$
		has a fixed point on each $T$-invariant ribbon. Ribbons entering a
		$T$-pair do not contain fixed points: such ribbons, if not disjoint,
		intersect only at the boundary of $M$ where $T$ has no fixed points.
		
		For any $i$, $1 \le i \le 2n$, denote by $\bar i \bydef i + n \bmod
		2n$. 
		
		\begin{definition}\label{Def:B-RibDecom}
			A {\em B-ribbon decomposition} with the profile
			$(\lambda,\mu,m,\ell)$ is a DBS $M$ with $2n$ marked points equipped
			with a smooth orientation-preserving involution $T: M \to M$ and a
			$T$-invariant ribbon decomposition $R$ such that
			\begin{enumerate}
				\item $T$ sends the marked point $a_i$ to the marked point $a_{\bar
					i}$, for all $i = 1 \DT, 2n$.
				
				\item $\Lambda_+(M,T) = \lambda$, $\Lambda_-(M,T) = \mu$.
				
				\item The number of $T$-invariant ribbons in the decomposition $R$
				is $\ell$ and the number of $T$-pairs of ribbons is $m$.
				
				\item $T$ sends the ribbon $\rho_k$ either to itself or to the
				ribbon $\rho_{k\pm1}$, for all $k = 1 \DT, 2m+\ell$.
				
				\item $T$ has exactly one fixed point inside each $T$-invariant
				ribbon and no fixed points elsewhere (in particular, it has no
				fixed points on the boundary of $M$).
			\end{enumerate}
		\end{definition}
		
		The total space $M$ of a B-ribbon decomposition is thus a DBS with a
		ribbon decomposition $H_1 H_2 \dots H_{m+\ell} E_{2n}$ where each
		$H_k$ is either an operation $G[i_k,\bar i_k]$ of gluing a
		$T$-invariant ribbon or an operation $G[i,j] G[\bar i,\bar j]$ of
		gluing a $T$-pair.
		
		B-ribbon decompositions are split into equivalence classes in a
		natural way: $(M_1,T_1,R_1)$ is said to be equivalent to
		$(M_2,T_2,R_2)$ if there exists an orientation-preserving
		diffeomorphism $f: M_1 \to M_2$ mapping marked points of $M_1$ to the
		marked points of $M_2$ with the same numbers, each ribbon of $R_1$, to
		the ribbon of $R_2$ with the same number, and transforming the
		involutions one into the other: $f \circ T_1 = T_2 \circ f$.
		
		Let $\xi$ be a quotient of a ribbon $\rho$ by the symmetry with
		respect to its center. We call $\xi$ a {\em petal}; it is a bigon with
		a black and a white vertex; one of its sides is long, and the other,
		short. The image of the diagonal is a line joining the black vertex
		with an internal point of the petal.
		
		Let $M$ be a DBS with $n$ marked points, and $1 \le i \le n$. Denote
		by $G[i] M$ the result of gluing of a petal to $M$ such that the black
		vertex is identified with the marked point $a_i$, and the short side,
		with a short segment of $\dM$ containing $a_i$ and directed according
		to the orientation of $\dM$. Apparently, $G[i] M$ is a DBS, and its
		boundary inherits an orientation from $\dM$. We call $G[i]$ an
		operation of petal gluing; it is similar to the operation $G[i,j]$ of
		ribbon gluing described above. 
		
		Let $(M,T,R)$ be a B-ribbon decomposition with the profile
		$(\lambda,\mu,m,\ell)$. Denote by $\overline M$ the quotient of $M$ by
		the involution $T$. Since $T$ has no fixed points on the boundary and
		$T(a_i) = a_{\bar i}$, the quotient $\overline M$ has a
		natural structure of a DBS with $n$ marked points $b_i = p(a_i) =
		p(a_{\bar i})$, $i = 1 \DT, n$, where $p: M \to \overline M$ is
		the natural projection. If the ribbons $\rho_k$ and $\rho_{k+1}$ form a
		$T$-pair then $\rho = p(\rho_1) = p(\rho_2) \subset \overline M$ is a
		ribbon; if the ribbon $\rho_k$ is $T$-invariant then $p(\rho_k)
		\subset \overline M$ is a petal. In other words, if $M = H_1 \dots
		H_{m+\ell} E_{2n}$ is a ribbon decomposition described above then
		$\overline M = \overline H_1 \dots \overline H_{m+\ell} E_n$ where
		$\overline H_k = G[i_k,j_k]$ if $H_k = G[i_k, j_k] G[\bar i_k, \bar
		j_k]$ and $\overline H_k = G[i_k]$ if $H_k = G[i_k, \bar i_k]$. 
		
		\begin{theorem}\label{Th:1-1Corresp}
			Let $\lambda$ and $\mu$ be two partitions such that
			$\lmod\lambda\rmod + \lmod\mu\rmod = n$, and $m$, $\ell$, two
			positive integers. There is a one-to-one correspondences between
			sequences of reflections in the group $B_n$ having profile
			$(\lambda,\mu,m,\ell)$ and equivalence classes of B-ribbon
			decompositions having the same profile.
		\end{theorem}
		
		\begin{proof}
			The correspondence in question relates a sequence of reflections $h_1
			\DT, h_{m+\ell}$ to a ribbon decomposition $H_1 \dots H_{m+\ell}
			E_{2n}$ of the DBS $M$ where for every $k$ one has $H_k = G[i,j]G[\bar
			i, \bar j]$ if $h_k = (ij)(\bar i \bar j)$ and $H_k = G[i,\bar i]$
			if $h_k = (i \bar i)$. The involution $T$ exchanges the two ribbons
			attached in the first case and leaves the ribbon invariant in the
			second case. By \cite[Theorem 1.11]{TwistedHurwitz} the boundary
			permutation of $M$ is equal to $h_1 \dots h_{m+\ell} \in B_n$. It
			follows from the description of the conjugacy classes in Section
			\ref{SSec:ConjClass} that $h_1 \dots h_{m+\ell} \in
			C_{\lambda\mid\mu}$ if and only if $\Lambda_+(M,T) = \lambda$ and
			$\Lambda_-(M,T) = \mu$.
		\end{proof}
		
		\subsection{Fermionic interpretation of B-cut-and-join}\label{SSec:Fermion}
		
		First, recall some classical results about Fock space,
		a.k.a.\ (semi-)infinite wedge product. For a partition $\lambda =
		(\lambda_1 \DT\ge \lambda_s)$ define a {\em Fock set}
		\begin{equation*}
			\Phi_\lambda \bydef \{i - \frac12 - \lambda_i \mid i = 1, 2, \dots\}
			\subset \HInteger;
		\end{equation*}
		here one assumes $\lambda_i = 0$ if $i > s$.
		
		\begin{lemma}
			\begin{enumerate}
				\item $\#\bigl(\Phi_\lambda \cap \{x \mid x < 0\}\bigr) =
				\#\bigl(\{x \mid x > 0\} \setminus \Phi_\lambda =
				\lmod\lambda\rmod \bydef \lambda_1 \DT+ \lambda_s$.
				
				\item If a subset $\Phi \subset \HInteger$ is such that $\Phi \cap \{x
				\mid x < 0\}$ and $\{x \mid x > 0\} \setminus \Phi$ are finite sets
				of equal cardinality then there exists a unique partition $\lambda$
				such that $\Phi = \Phi_\lambda$.
			\end{enumerate}
		\end{lemma}
		
		The lemma is proved by direct induction; see \cite{FSpace} for
		details.
		
		A classical (charge zero) Fock space \cite{FSpace} (or fermion state
		space) is a vector space $\Fock$ with the basis $w_\lambda$ where
		$\lambda$ runs through all partitions, including the empty partition;
		$w_\emptyset$ is called the vacuum vector. One usually represents
		$w_\lambda$ as an ``infinite wedge product'' $\bigwedge_{u \in
			\Phi_\lambda} z_u = z_{1/2-\lambda_1} \wedge z_{3/2-\lambda_2}
		\wedge \dots$; note that the indices at the product form an increasing
		sequence.
		
		Let $k \in \Integer$, $k \ne 0$. Consider a linear operator $a_k:
		\Fock \to \Fock$ acting on the basic vectors as follows:
		\begin{equation*}
			a_k (\bigwedge_{u \in \Phi_\lambda} z_u) = \sum_{v \in \Phi_\lambda}
			z_{v-k} \wedge (-1)^{d_{\lambda,k}(v)} \bigwedge_{u \in \Phi_\lambda
				\setminus \{v\}} z_u
		\end{equation*}
		where $d_{\lambda,k}(v)$ is defined as $\#\{u \in \Phi_\lambda \mid
		v-k \le u < v\}$ if $k > 0$ and as $\#\{u \in \Phi_\lambda \mid v < u
		\le v-k\}$ if $k < 0$.
		
		\begin{theorem}[\cite{FSpace}]\label{Th:Fock}
			\begin{enumerate}
				\item Operators $a_i$ and $a_j$ commute unless $i+j = 0$.
				
				\item\label{It:BFC} For any element $x \in \Fock$ there exists a
				unique polynomial $R_x \in \Complex[p] \bydef \Complex[p_1, p_2,
				\dots]$ such that $x = R_x(a_1,a_2,\dots)w_{\emptyset}$. If $x =
				w_\lambda$ then $R_x = s_\lambda$ (a classical Schur polynomial).
			\end{enumerate}
		\end{theorem}
		
		The map $R: \Fock \to \Complex[p]$ defined in \ref{It:BFC} is called
		boson-fermion correspondence. It is a linear bijection, so for any
		linear operator $A: \Fock \to \Fock$ one can define its polynomial
		counterpart $\hat A \bydef R A R^{-1}: \Complex[p] \to \Complex[p]$;
		the map $A \leftrightarrow \hat A$ is an isomorphism of the algebras
		of linear operators on $\Fock$ and on $\Complex[p]$.
		
		\begin{proposition}[\cite{FSpace}]\label{Pp:OperHat}
			\begin{enumerate}
				\item\label{It:PdP} If $A = a_k$ then $\hat A = p_k$ (the operator of
				multiplication) if $k > 0$ and $\hat A = -k \pder{}{p_{-k}}$ if $k
				< 0$.
				
				\item\label{It:Energy} If $A = \sum_{i=1}^\infty a_i a_{-i}$ (called
				the energy operator) then $\hat A$ is the Euler vector field
				$\sum_{i=1}^\infty i p_i \pder{}{p_i}$.
				
				\item\label{It:CJ} If $A = \sum_{\substack{i,j,k \ne 0\\ i+j+k = 0}}
				a_i a_j a_k$ then $\hat A$ is the classical cut-and-join operator
				$\CJ$.
			\end{enumerate}
		\end{proposition}
		
		It can be proved (see \cite{FSpace}) that for any partition $\lambda$
		the basic element $w_\lambda$ is an eigenvector of both the energy
		operator and the operator from assertion \ref{It:CJ} of Proposition
		\ref{Pp:OperHat}. The corresponding eigenvalues are $\lmod
		\lambda\rmod \bydef \sum_i \lambda_i$ and $\phi(\lambda) \bydef \sum_i
		\lambda_i(\lambda_i-2i+1)$. Equivalently, the Schur polynomial
		$s_\lambda$ is the eigenvector of the Euler field with the eigenvalue
		$\lmod \lambda\rmod = \deg s_\lambda$ and of the cut-and-join operator
		with the eigenvalue $\phi(\lambda)$.
		
		Consider now a space $L \bydef \Fock \otimes \Fock$ (two-particle
		fermion state space) with the basis $w_\lambda \otimes w_\mu$, where
		$\lambda, \mu$ are partitions. The tensor square of the boson-fermion
		correspondence $R^{\otimes 2}: L \to \Complex[p]^{\otimes 2}
		=\Complex[u,v]$ where $u = (u_1, u_2, \dots)$ and $v = (v_1, v_2,
		\dots)$ is a linear isomorphism. Like for the classical case, any
		operator $A: L \to L$ has its polynomial counterpart $R^{\otimes 2} A
		(R^{\otimes 2})^{-1}: \Complex[u,v] \to \Complex[u,v]$, which we also
		denote by $\hat A$.
		
		For every $k \ne 0$ consider two linear operators: $b_k \bydef a_k
		\otimes I + I \otimes a_k$ and $c_k \bydef a_k \otimes I - I \otimes
		a_k$. It follows from Proposition \ref{Pp:OperHat} (assertion
		\ref{It:PdP}) that
		\begin{align}
			&\begin{aligned}
				&\hat b_k = u_k + v_k = p_k,\\
				&\hat c_k = u_k - v_k = q_k
			\end{aligned}
			\qquad \text{if } k > 0,\label{Eq:KPos}\\
			&\begin{aligned}
				&\hat b_k = -k(\pder{}{u_{-k}} + \pder{}{v_{-k}}) =
				-2k\pder{}{p_{-k}},\\
				&\hat c_k = -k(\pder{}{u_{-k}} - \pder{}{v_{-k}}) =
				-2k\pder{}{q_{-k}}
			\end{aligned}
			\qquad \text{if } k < 0.\label{Eq:KNeg}
		\end{align}
		Here we are using notation from Section \ref{SSec:Change}.
		
		Apparently, $b_i$ and $b_j$, as well as $c_i$ and $c_j$, commute unless
		$i+j = 0$; $b_i$ and $c_j$ commute for all $i, j$.
		
		\begin{proposition}
			If 
			\begin{align*}
				J_1 &\bydef \frac{1}{12}(\sum_{\substack{i+j+k=0\\i,j,k \ne 0}}
				b_ib_jb_k + \sum_{\substack{i+j+k=0\\i,j,k \ne 0}}b_ic_jc_k +
				\sum_{\substack{i+j+k=0\\i,j,k\neq0}}
				c_ib_jc_k+\sum_{\substack{i+j+k=0\\i,j,k \ne 0}}c_ic_jb_k)\\
				&= \frac{1}{12}\sum_{\substack{i+j+k=0\\i,j,k\neq0}} b_ib_jb_k +
				\frac{1}{4} \sum_{\substack{i+j+k=0\\i,j,k \ne 0}} b_ic_jc_k,\\
				J_2 &\bydef \sum_{\ell=1}^\infty b_\ell c_{-\ell} +
				\sum_{\ell=1}^\infty c_\ell b_{-\ell}
			\end{align*}
			then $\hat{J_1} = \CJ_1$ and $\hat{J_2} = \CJ_2$ (the B-cut-and-join
			operators), and also
			\begin{align*}
				J_1 &= C \otimes \name{Id} + \name{Id} \otimes C,\\
				J_2 &= E\otimes \name{Id} - \name{Id}\otimes E
			\end{align*}
			where $E$ and $C$ are the operators defined in Proposition
			\ref{Pp:OperHat} (2) and (3) respectively.
		\end{proposition}
		
		The proof is straightforward: the first pair of equations is proved by
		a calculation using \eqref{Eq:KPos}, \eqref{Eq:KNeg}, \eqref{Eq:CJ_1}
		and \eqref{Eq:CJ_2}, as well as the commutativity relations. The
		second pair of equations follows from Proposition \ref{Pp:ChVar} or
		can be obtained by a direct calculation, too.
		
		\begin{corollary}\label{Th:Fock2}
			$w_\lambda\otimes w_\mu$ are eigenvectors of $J_1$ and $J_2$ with the
			eigenvalues $\sum_{i=1}^\infty (\lambda_i(\lambda_i-2i+1) +
			\mu_i(\mu_i-2i+1)) $ and $\sum_{i=1}^\infty (\lambda_i-\mu_i)$,
			respectively.
		\end{corollary}
		
		\subsection{KP hierarchy for B-Hutwitz numbers}
		
		The KP hierarchy is one of the best known and studied integrable
		systems; see e.g.\ \cite{Kramer,KazarianHodge,FSpace} for details. It
		is an infinite system of PDE applied to a formal series $F \in
		\Complex[[t]]$ where $t = (t_1, t_2, \dots)$ is a countable collection
		of variables (``times''). Each equation looks like
		$\pdertwo{F}{t_i}{t_j} = P_{ij}(F)$ where $P_{ij}$ is a homogeneous
		multivariate polynomial with partial derivatives of $F$ used as its
		arguments. For example, the first two equations of the hierarchy are
		
		\begin{align*}
			&F_{22}=-\frac{1}{2}F_{11}^2 + F_{31} -\frac{1}{12}F_{1111}\\
			&F_{32}=-F_{11}F_{21}+F_{41}-\frac{1}{6}F_{2111};
		\end{align*}
		here $F_{i_1i_2\dots i_n}$ means $\frac{\partial^nF}{\partial
			u_{i_1}\partial u_{i_2\dots\partial u_{i_n}}}$. If $F$ is a solution
		of the hierarchy, its exponential $\tau= e^F$ is called a
		$\tau$-function.
		
		\begin{example}\label{Ex:T1Tau}
			Equations of the hierarchy do not contain first-order derivatives,
			so the function $F(t) = t_1$ is a solution of the KP hierarchy;
			$\tau = e^{t_1}$ is a $\tau$-function.
		\end{example}
		
		\begin{proposition}\label{Pp:Sato}
			There exists a Lie algebra $\mathcal G$ of differential operators on
			the space $\Complex[[p]]$ having the following properties:
			\begin{enumerate}
				\item\label{It:GLInf} $\mathcal G$ is isomorphic to a
				one-dimensional central extension $\widehat{\name{gl}(\infty)}$ of
				a Lie algebra $\name{gl}(\infty)$ of infinite matrices
				$(x_{ij})_{i,j \in \Integer}$ having nonzero elements at a finite
				set of diagonals.
				
				\item\label{It:F2F3} The Euler field $E$ and the cut-and-join
				operator $\CJ$ belong to $\mathcal G$.
				
				\item\label{It:Orbit} For any $\tau$-function $f$ and any $X \in
				\mathcal G$, the power series $e^X f$ is a $\tau$-function,
				too. 
			\end{enumerate}
		\end{proposition}
		
		See \cite{Kramer} for the definition of $\mathcal G$, \cite{2Sato} and
		\protect{\cite[section 6]{KazarianHodge}} for the proof.
		
		\begin{theorem}\label{Th:HBsolKP}
			The generating function $\mathcal{H}^B(\beta,\gamma,u+v,u-v)$ is a
			$2$-parameter family of $\tau$-functions, independently in the $u$
			and the $v$ variables.
		\end{theorem} 
		
		\begin{proof}
			Equation \eqref{Eq:GenFunc} and Proposition \ref{Pp:ChVar} imply that
			\begin{align*}
				\mathcal{H}^B(\beta,\gamma,u+v,u-v) = e^{\beta(\CJ_u+
					\CJ_v)}e^{\gamma(E_u-E_v)} e^{u_1}e^{v_1} =
				e^{(\beta\CJ_{u}+\gamma E_u)}e^{u_1} e^{(\beta\CJ_{v}-\gamma
					E_v)}e^{v_1}
			\end{align*}
			(operators $\CJ_i$ and $\CJ_v$ commute because they act on different
			sets of variables).
			
			By assertion \ref{It:F2F3} of Proposition \ref{Pp:Sato},
			$\beta\CJ_{u}+ \gamma E_u \in \mathcal G_u$ (the Lie algebra $\mathcal
			G$ acting by differential operators on the $u$ variables). By Example
			\ref{Ex:T1Tau}, $e^{u_1}$ is a $\tau$-function (in the $u$ variables),
			and so is $\mathcal{H}^B(\beta,\gamma,u+v,u-v)$ by assertion
			\ref{It:Orbit} of the same proposition. The reasoning for the $v$
			variables is the same.
		\end{proof}
		
		\begin{corollary}
			The generating function $\mathcal{H}^D(\beta,u+v,u-v)$ is a 1-parameter
			family of $\tau$-functions independently in the $u$ and the $v$
			variables.
		\end{corollary}
		\begin{proof}
			We apply the same strategy as in the proof of theorem
			\ref{Th:HBsolKP} to the function $\mathcal{H}^D(\beta,u+v,u-v) =
			e^{\beta(\CJ_u+ \CJ_v)}e^{u_1}e^{v_1}$.
		\end{proof}

	\end{document}